\documentclass[11pt]{amsart}

\usepackage[T1]{fontenc}
\usepackage{lmodern}
\usepackage{microtype}
\usepackage{amsmath,amssymb,amsthm,mathtools}
\usepackage{enumitem}
\usepackage[hidelinks]{hyperref}

\newtheorem{theorem}{Theorem}[section]
\newtheorem{lemma}[theorem]{Lemma}
\newtheorem{proposition}[theorem]{Proposition}
\newtheorem{corollary}[theorem]{Corollary}
\theoremstyle{definition}

\theoremstyle{remark}
\newtheorem{remark}[theorem]{Remark}

\newcommand{\R}{\mathbb{R}}
\newcommand{\C}{\mathbb{C}}

\newcommand{\Z}{\mathbb{Z}}
\newcommand{\GL}{\operatorname{GL}}
\newcommand{\Aut}{\operatorname{Aut}}
\newcommand{\Hom}{\operatorname{Hom}}
\newcommand{\vol}{\operatorname{vol}}
\newcommand{\Int}{\operatorname{int}}
\newcommand{\Sym}{\operatorname{Sym}}
\newcommand{\Lip}{\operatorname{Lip}}

\newcommand{\supp}{\operatorname{supp}}

\title[The Complex Banach Isometric Conjecture]{The Complex Banach Isometric Conjecture}

\author{Xinan Dai}
\thanks{Xinan Dai is currently a Ph.D. student at Fudan University and a visiting student at the AI for Scientific Simulation and Discovery Lab, Westlake University.}
\address{College of Future Information and Technology,
Fudan University, Shanghai, China}
\curraddr{Department of Artificial Intelligence,
School of Engineering, Westlake University, Hangzhou, China}
\email{xndai23@m.fudan.edu.cn}

\author{Wenhao Deng}
\thanks{Wenhao Deng is a student at the University of Glasgow and is currently an intern at the AI for Scientific Simulation and Discovery Lab, Westlake University.}
\address{University of Glasgow, Glasgow, United Kingdom}
\curraddr{Department of Artificial Intelligence,
School of Engineering, Westlake University, Hangzhou, China}
\email{dengwenhao@westlake.edu.cn}

\author{Yingdong Shi}
\address{School of Information Science and Technology, ShanghaiTech University, Shanghai, China}
\email{shiyd2023@shanghaitech.edu.cn}

\author{Tailin Wu}
\address{Department of Artificial Intelligence,
School of Engineering, Westlake University, Hangzhou, China}
\email{wutailin@westlake.edu.cn}

\author{Yuchen Yang}
\address{Department of Artificial Intelligence,
School of Engineering, Westlake University, Hangzhou, China}
\email{yangyuchen@westlake.edu.cn}

\date{August 15, 2026}

\begin{document}

\begin{abstract}
Banach asked whether a normed space must be Hilbert if, for one fixed dimension greater than one, all subspaces of that dimension are linearly isometric.  We prove the complex case.  By the codimension-one reduction, the essential finite-dimensional problem is to characterize a balanced convex body in a complex $(n+1)$-space whose complex hyperplane sections are all complex-linearly equivalent.  We show that such a body is a Hermitian ellipsoid.  The proof adapts the recent bundle--degree method of Lu and Yang for the real problem, with two features specific to the complex setting.  After covariance normalization, write $G=\Aut_{\C}(S)$ for the complex-linear symmetry group of a model section $S$.  The exact section maps form a principal $G$-bundle over $S^{2n+1}$; reduction to $G^{\circ}$ places its obstruction in $\pi_{2n}(G^{\circ})$, which is finite.  Pulling back by a suitable positive-degree self-map trivializes this bundle and yields a global Lipschitz family of exact complex-linear section maps.  Brouwer degree and a signed degree formula then imply that $p_S^{2n+2}$ and $p_S^{2n+4}$ are real homogeneous polynomials, where $p_S$ is the norm of a model section.  Unique factorization forces $p_S^2$ to be quadratic, and phase invariance makes the resulting quadratic form Hermitian.  The parallelogram identity then yields the general complex Banach-space statement.
\end{abstract}

\maketitle

\section{Introduction}\label{sec:intro}

Banach asked in 1932 whether the linear isometry type of one family of finite-dimensional subspaces can determine the norm of the whole space \cite[Chapter~XII, Remarks, property~(5), p.~244]{Banach1932}.  In the form considered here, the question is the following.

\begin{quote}
\emph{Let $X$ be a real or complex Banach space.  If, for some fixed integer $2\leq n<\dim X$, all $n$-dimensional linear subspaces of $X$ are linearly isometric, must $X$ be a Hilbert space?}
\end{quote}

A basic reduction explains why hyperplane sections are the natural endpoint.  Suppose the assertion is known whenever $\dim X=n+1$.  Given arbitrary $x,y\in X$, choose an $(n+1)$-dimensional subspace containing them.  Its $n$-dimensional hyperplanes are mutually isometric by hypothesis, so the codimension-one result makes the restricted norm Hilbertian.  The parallelogram identity therefore holds for $x$ and $y$, and hence on all of $X$.  Thus the problem is already decided by the case of mutually equivalent hyperplane sections of one finite-dimensional unit ball.

The known results progressively narrowed this endpoint.  Auerbach, Mazur and Ulam settled the real case $n=2$ \cite{AuerbachMazurUlam1935}.  Dvoretzky's theorem gives the infinite-dimensional real case \cite{Dvoretzky1959}, and Milman's complex form of the theorem gives the corresponding infinite-dimensional complex statement \cite{Milman1971}.  In finite dimensions, Gromov proved the conjecture for even $n$ over both fields; for odd $n$ he also treated real spaces when $\dim X\geq n+2$ and complex spaces when $\dim_{\C}X\geq2n$ \cite{Gromov1967}.  Later, Bor, Hern\'andez-Lamoneda, Jim\'enez-Desantiago and Montejano treated real $n\equiv1\pmod4$, with the possible exception $n=133$ \cite{BorEtAl2021}, and Ivanov, Mamaev and Nordskova settled the four-dimensional ambient real case \cite{IvanovMamaevNordskova2023}.  For complex spaces, Bracho and Montejano proved the conjecture for $n\equiv1\pmod4$ \cite{BrachoMontejano2021}; related characterizations of complex ellipsoids and complex symmetry were developed by Arocha, Bracho and Montejano \cite{ArochaBrachoMontejano2023}.  Consequently, after Gromov's high-dimensional result, the finite-dimensional complex cases not covered by these theorems occur for
\[
 n\equiv3\pmod4,
 \qquad
 n<\dim_{\C}X<2n.
\]
By the preceding reduction, it is enough to understand the codimension-one member $\dim_{\C}X=n+1$ of this range.

There have also been two very recent developments on the real problem.  Zhang has posted a preprint claiming a complete finite-dimensional solution by a different route based on global linear maps between lower-dimensional sections \cite{Zhang2026}.  Lu and Yang subsequently proved the remaining odd-dimensional real cases by a bundle--degree argument, completing the real conjecture together with Gromov's even-dimensional theorem \cite{LuYang2026}.  We use their globalization mechanism as the starting point rather than as a claimed new ingredient.  The complex contribution is the dimension shift that places the obstruction in an even homotopy group for every $n$, together with the complex-linear normalization and the final recovery of Hermitian structure.

To see the topological shift, let $V$ be a complex $(n+1)$-space and parametrize its complex hyperplanes by unit normals $u\in S^{2n+1}$.  After normalizing a model section by its covariance, its complex-linear symmetry group $G$ is a compact subgroup of a unitary group.  Reducing the section-isometry bundle to the identity component $G^{\circ}$ gives a class
\[
 [\mathcal P^{\circ}]\in\pi_{2n+1}(BG^{\circ})\cong\pi_{2n}(G^{\circ}).
\]
The group on the right is finite because $2n>0$ is even.  Hence the obstruction is torsion for every $n$.  A self-map of $S^{2n+1}$ of suitable positive degree kills it after pullback.  This is the reason the same globalization principle that requires odd section dimension in the real proof is available uniformly in the complex setting.

Two further points are genuinely complex.  First, the covariance normalization and all local and global section maps must remain complex linear.  For a balanced body, its real covariance commutes with multiplication by $i$, so its positive square root does as well.  Second, the degree argument lives on the underlying real spaces and initially produces a real quadratic form.  The invariance $p(ix)=p(x)$ then forces the associated real bilinear form to be $J$-invariant, which is precisely what is needed to recover a Hermitian form.

Our finite-dimensional statement is the following.

\begin{theorem}[complex hyperplane theorem]\label{thm:hyperplane}
Let $n\geq2$, let $V$ be a complex vector space of dimension $n+1$, and let $K\subset V$ be the closed unit ball of a complex norm.  Suppose that for every pair of complex hyperplanes $H_1,H_2\subset V$ there is a complex-linear isomorphism $T:H_1\to H_2$ such that
\[
 T(K\cap H_1)=K\cap H_2.
\]
Then $K$ is the unit ball of a positive-definite Hermitian form.
\end{theorem}

The codimension-one reduction gives the Banach-space formulation immediately.

\begin{corollary}[complex Banach isometric conjecture]\label{cor:main}
Let $X$ be a complex Banach space.  If, for some fixed integer $2\leq n<\dim_{\C}X$, all complex $n$-dimensional linear subspaces of $X$ are mutually complex-linearly isometric, then $X$ is a complex Hilbert space.
\end{corollary}

The proof has three steps, each with a different role.  Fix a model hyperplane $E$ and a normalized model section $S\subset E$.  A finite collection of moment tensors detects the compact group
\[
 G=\Aut_{\C}(S)
\]
exactly.  These moments provide local coordinates on the orbit $U(E)/G$, and hence convert the pointwise existence of section isometries into a Lipschitz principal $G$-bundle over $S^{2n+1}$.  This first step is analytic rather than merely formal: the later Jacobian argument requires a family that is both exact and Lipschitz.

The second step is topological.  After reducing to $G^{\circ}$, a positive-degree pullback trivializes the bundle.  A continuous trivializing section is then regularized inside an associated vector bundle and retracted equivariantly to the exact orbit, producing a Lipschitz family $A_z$ with
\[
 A_z(S)=K\cap\phi(z)^{\perp}.
\]
The original bundle need not be trivial.  What matters is only that its obstruction is torsion: the pullback converts that obstruction into the integer $\deg\phi$, and this same integer reappears as the multiplicity in the degree formula below.

The final step turns this topological multiplicity into algebraic rigidity.  For fixed $0\neq y\in E$, the map $z\mapsto A_zy$ has degree $\deg\phi$ on the appropriate boundary.  Its homogeneous extension therefore has constant Brouwer degree on the interior of $p_S(y)K$.  With $m=2n+2$, the signed degree formula yields
\[
 p_S(y)^{m+2j}\in\R[E_{\R}]
 \qquad(j\geq0).
\]
Only $j=0,1$ are needed: they say that two consecutive powers of $p_S^2$ are polynomials.  Unique factorization then forces $p_S^2$ itself to be a quadratic polynomial.  This is the algebraic bottleneck of the proof; no explicit formula for the section maps is required once these two adjacent powers are known.  Phase invariance turns the quadratic form into a Hermitian one, and the parallelogram law finishes the argument.

Section~\ref{sec:prelim} records the convex-geometric, bundle-theoretic and degree-theoretic ingredients.  Section~\ref{sec:bundle} constructs and globalizes the exact complex isometry bundle.  Section~\ref{sec:degree} converts the global family into polynomial identities.  Section~\ref{sec:hermitian} proves quadratic and Hermitian rigidity, and Section~\ref{sec:ambient} returns to arbitrary complex Banach spaces.

\section{Preliminaries}\label{sec:prelim}

Throughout the finite-dimensional argument, $V$ is a complex vector space of dimension $n+1$, $n\geq2$, equipped with a fixed auxiliary Hermitian inner product $\langle\cdot,\cdot\rangle$, linear in the first variable.  We write $|v|^2=\langle v,v\rangle$, let $V_{\R}$ denote the underlying real Euclidean space, and write $Jv=iv$.  The associated real inner product is $\langle x,y\rangle_{\R}=\operatorname{Re}\langle x,y\rangle$.

If $T$ is an origin-symmetric convex body, $p_T$ denotes its Minkowski functional and $\rho_T(\theta)=p_T(\theta)^{-1}$ its radial function on the auxiliary unit sphere.  A subset of a complex vector space is \emph{balanced} if $e^{it}T=T$ for every $t\in\R$.  The unit ball of a complex norm is balanced, and therefore $JT=T$ and $p_T(e^{it}x)=p_T(x)$.

The finite-moment, bundle and degree constructions below follow the architecture of Lu and Yang \cite[Sections~2--4]{LuYang2026}, but are written in a form that preserves complex linearity.  We include the needed arguments because the changes of scalar field enter at several structural points rather than only in the final polarization step.

\subsection{Covariance and balanced normalization}

For a convex body $T$ in a real Euclidean $m$-space define its normalized covariance operator by
\[
C(T)=\frac{1}{\vol_m(T)}\int_T x\otimes_{\R}x\,dx,
\qquad
(x\otimes_{\R}x)y=\langle y,x\rangle_{\R}x.
\]
It is positive definite whenever $T$ has nonempty interior.

\begin{lemma}[linear covariance]\label{lem:covariance}
If $A$ is an invertible real-linear map, then
\[
C(AT)=AC(T)A^*.
\]
If, in addition, $T$ is balanced in a complex Euclidean space, then $C(T)J=JC(T)$.  Consequently $C(T)^{1/2}$ and $C(T)^{-1/2}$ are complex linear.
\end{lemma}

\begin{proof}
The first identity follows from $x=Ay$ in the defining integral: the factor $|\det_{\R}A|$ occurs in both the numerator and the volume and cancels.  If $JT=T$, changing variables by the orthogonal map $J$ gives
\[
C(T)=JC(T)J^*.
\]
Since $J^*=-J$, this is equivalent to $C(T)J=JC(T)$.  The positive square root and its inverse are functions of $C(T)$ and therefore commute with $J$; see, for example, \cite[Chapter~1]{Bhatia2007}.
\end{proof}

\begin{lemma}[unitary normalization]\label{lem:compactgroup}
Let $E$ be a complex $n$-space and let $S\subset E$ be the unit ball of a complex norm with $C(S)=I_E$.  Then
\[
G:=\Aut_{\C}(S)=\{g\in\GL_{\C}(E):gS=S\}
\]
is a compact subgroup of $U(E)$.
\end{lemma}

\begin{proof}
For $g\in G$, Lemma~\ref{lem:covariance} gives
\[
I_E=C(S)=C(gS)=gg^*.
\]
Thus $g$ is unitary.  The stabilizer of the compact set $S$ is closed in $U(E)$, hence compact.
\end{proof}

We shall use the following elementary regularity estimate repeatedly.  If $aB\subset T\subset bB$ for the auxiliary Euclidean unit ball $B$, then
\begin{equation}\label{eq:gauge-lip}
|p_T(x)-p_T(y)|\leq a^{-1}|x-y|,
\qquad
|\rho_T(\theta)-\rho_T(\eta)|\leq \frac{b^2}{a}|\theta-\eta|.
\end{equation}
The first estimate follows from subadditivity of $p_T$; the second follows by applying the first to $1/p_T$ on the unit sphere, where $p_T\geq b^{-1}$.

\subsection{Finite moments and orbit coordinates}\label{subsec:moments}

For $j\geq1$, let $\Sym^j(E_{\R})$ be the $j$th symmetric tensor power and set
\[
M_j(T)=\frac{1}{\vol(T)}\int_T x^{\otimes j}\,dx\in\Sym^j(E_{\R}).
\]
The action of a real-linear map $A$ on $M_j(T)$ is the induced action $A^{\otimes j}$, and a change of variables gives
\begin{equation}\label{eq:moment-transform}
M_j(AT)=A^{\otimes j}M_j(T).
\end{equation}
In polar coordinates, both $\vol(T)$ and each fixed $M_j(T)$ depend locally Lipschitzly on $\rho_T$ in the uniform norm on any family with common inner and outer Euclidean radii.

The next lemma is the finite-dimensional algebraic device that turns exact symmetries of a body into finitely many smooth coordinates.  It is the unitary version of the moment construction used in \cite[Section~3.1]{LuYang2026}; we include the argument because preserving complex linearity is essential below.

\begin{lemma}[finite moment detection]\label{lem:moment-detect}
Let $S\subset E$ satisfy $C(S)=I_E$, and let $G=\Aut_{\C}(S)\subset U(E)$.  There are integers $j_1,\dots,j_s\geq1$ such that the stabilizer in $U(E)$ of
\[
\mathbf M(S)=\bigl(M_{j_1}(S),\dots,M_{j_s}(S)\bigr)
\]
is exactly $G$.
\end{lemma}

\begin{proof}
Suppose first that $Q\in U(E)$ fixes $M_j(S)$ for every $j$.  Then the normalized Lebesgue measures of $S$ and $QS$ have the same integral against every real polynomial on $E_{\R}$.  Polynomials are uniformly dense in the continuous functions on the compact set $S\cup QS$ by Stone--Weierstrass \cite[Theorem~7.32]{Rudin1976}.  Hence the two normalized measures agree.  Their supports are $S$ and $QS$, so $QS=S$ and $Q\in G$.

Choose real coordinates on the matrix space containing $U(E)$.  Every coordinate of
\[
Q^{\otimes j}M_j(S)-M_j(S)
\]
is a polynomial in the real and imaginary parts of the entries of $Q$.  Let $I$ be the ideal generated by all these coordinate polynomials as $j$ varies.  By the Hilbert basis theorem, the ambient real polynomial ring is Noetherian \cite[Corollary~7.6]{AtiyahMacdonald1969}.  Choose generators $f_1,\dots,f_N$ of $I$.  Each $f_\nu$ is a finite polynomial combination of the original moment equations; collecting the finitely many original equations that occur in those representations gives an ideal $J$ with $(f_1,\dots,f_N)\subseteq J\subseteq I$.  Hence $J=I$.  Grouping this finite collection by moment degree gives $j_1,\dots,j_s$.  Their common stabilizer inside $U(E)$ is therefore the common stabilizer of all moments, which is $G$ by the first paragraph.
\end{proof}

Let
\[
\mathcal W=\bigoplus_{i=1}^s\Sym^{j_i}(E_{\R}).
\]
The compact Lie group $U(E)$ acts orthogonally on $\mathcal W$.

\begin{lemma}[orbit coordinates]\label{lem:orbit}
The map
\[
\Theta:U(E)/G\longrightarrow U(E)\mathbf M(S),
\qquad QG\longmapsto Q\mathbf M(S),
\]
is a diffeomorphism onto a compact embedded submanifold of $\mathcal W$.  Its inverse is locally Lipschitz with respect to the ambient Euclidean distance on $\mathcal W$.
\end{lemma}

\begin{proof}
The group $G$ is a closed Lie subgroup of $U(E)$, so $U(E)/G$ is a smooth compact manifold.  The map $\Theta$ is smooth and injective by Lemma~\ref{lem:moment-detect}.  It remains to check that its differential is injective.

At the identity coset, let $X\in\mathfrak u(E)$ represent a tangent vector in the kernel.  If $\rho$ denotes the finite-dimensional tensor representation on $\mathcal W$, then
\[
d\rho(X)\mathbf M(S)=0.
\]
Consequently
\[
\rho(e^{tX})\mathbf M(S)=e^{t\,d\rho(X)}\mathbf M(S)=\mathbf M(S)
\]
for every $t$, and Lemma~\ref{lem:moment-detect} implies $e^{tX}\in G$.  Thus $X$ lies in the Lie algebra of $G$, so the tangent vector in $U(E)/G$ is zero.  Equivariance gives injectivity everywhere.  Since the domain is compact, the injective immersion is an embedding.  The inverse is smooth on the orbit and therefore locally Lipschitz in smooth coordinates; embeddedness compares those coordinates with the ambient Euclidean norm.
\end{proof}

\subsection{Bundles and the finite homotopy obstruction}\label{subsec:bundles}

We use right principal bundles.  A \emph{finite Lipschitz atlas} on a principal $L$-bundle $P\to Z$ consists of finitely many local sections $s_i:U_i\to P$ such that, with the convention
\[
s_j(z)=s_i(z)g_{ij}(z)\qquad (z\in U_i\cap U_j),
\]
each transition map $g_{ij}:U_i\cap U_j\to L$ is Lipschitz.  A global section $s$ is called \emph{Lipschitz} relative to such an atlas if the local coordinate maps $h_i$ defined by $s=s_i h_i$ are Lipschitz.  On the compact manifolds used below, different finite smooth background atlases give equivalent notions, so no metric choice will enter the argument.

Standard classification gives, for a path-connected Lie group $H$ and $q\geq2$,
\begin{equation}\label{eq:bundle-classification}
\{\text{principal $H$-bundles over }S^q\}/\cong
\;\simeq\;
\pi_q(BH)\cong\pi_{q-1}(H);
\end{equation}
see \cite[Chapter~4, Sections~8 and 12--13; Chapter~1, Theorem~5.3 and Chapter~4, Section~11]{Husemoller1994}.  Here there is no based-versus-free ambiguity: $H$ is path connected, hence $\pi_1(BH)\cong\pi_0(H)=0$, so $BH$ is simply connected.  If $\phi:S^q\to S^q$ has degree $d$, precomposition by $\phi$ acts as multiplication by $d$ on $\pi_q(BH)$; this is the usual action of $[S^q,S^q]\cong\Z$ on $\pi_q(BH)$, see \cite[Section~4.1 and Corollary~4.25]{Hatcher2002}.

We also need the following finiteness theorem.

\begin{lemma}[even homotopy of compact Lie groups]\label{lem:finite-homotopy}
If $H$ is a compact connected Lie group and $q>0$ is even, then $\pi_q(H)$ is finite.
\end{lemma}

\begin{proof}
There is a finite covering homomorphism $T\times H_{\mathrm{ss}}\to H$, with $T$ a torus and $H_{\mathrm{ss}}$ compact, connected, simply connected and semisimple; see \cite[Section~2.9]{BorelHirzebruch1958}.  Covering maps induce isomorphisms on homotopy groups in degrees at least two, while a torus has no higher homotopy.  The assertion therefore reduces to $H_{\mathrm{ss}}$, for which Serre's finiteness theorem gives finiteness in positive even degrees \cite[Chapitre~V, \S3, Corollaire~2]{Serre1953}.
\end{proof}

\subsection{Degree theory in the Lipschitz category}\label{subsec:degree-prelim}

Whenever a complex vector space is regarded as a real vector space, we use its canonical complex orientation; Euclidean spheres are oriented as boundaries of their unit balls.  We write $B^m$ for the closed Euclidean unit ball in an oriented real $m$-space and $S^{m-1}=\partial B^m$.  For a continuous map $h:S^{m-1}\to S^{m-1}$, $\deg h$ denotes its topological degree.  For a bounded open $\Omega\subset\R^m$, a continuous map $F:\overline\Omega\to\R^m$, and $v\notin F(\partial\Omega)$, we write $\deg(F,\Omega,v)$ for Brouwer degree.

The cone on $h$ is
\[
C_h(0)=0,
\qquad
C_h(t\theta)=t h(\theta),\quad 0<t\leq1.
\]
The boundary characterization of Brouwer degree gives
\begin{equation}\label{eq:cone-degree-standard}
\deg(C_h,\Int B^m,v)=\deg h
\qquad (v\in\Int B^m).
\end{equation}
For example, $(h-v)/|h-v|$ is homotopic to $h$ through $(h-tv)/|h-tv|$.

For the analytic step, we use the signed degree formula in the following form.  If $F:D\to\R^m$ is Lipschitz on an open neighborhood $D$ of $\overline\Omega$, then $F$ is differentiable almost everywhere by Rademacher's theorem \cite[Theorem~3.2]{EvansGariepy2015}, has the Lusin $N$ property, and, provided $F(\partial\Omega)$ is null,
\begin{equation}\label{eq:signed-degree}
\int_{\Omega}g(F(x))\det DF(x)\,dx
=
\int_{\R^m}g(v)\deg(F,\Omega,v)\,dv
\end{equation}
for every bounded compactly supported Borel function $g$; see \cite[Remark~5.26(ii) and Theorem~5.27]{FonsecaGangbo1995}.  When a Lipschitz map is given only on $\overline\Omega$, we shall extend it coordinatewise by McShane's theorem \cite{McShane1934} before applying \eqref{eq:signed-degree}.

\section{The exact complex isometry bundle}\label{sec:bundle}

We now assume the hypotheses of Theorem~\ref{thm:hyperplane}.  Fix an auxiliary unit vector $u_0\in V$, put
\[
E=u_0^{\perp},
\qquad
S_0=K\cap E,
\]
and replace the model section by its covariance normalization
\begin{equation}\label{eq:model-normalization}
S=C(S_0)^{-1/2}S_0.
\end{equation}
Lemma~\ref{lem:covariance} shows that the normalizing map is complex linear and $C(S)=I_E$.  Set $G=\Aut_{\C}(S)\subset U(E)$.

For $u\in S(V)=S^{2n+1}$ define
\begin{equation}\label{eq:fiber}
\mathcal P_u=
\left\{
A\in\Hom_{\C}(E,V):
A(E)=u^{\perp},\quad A(S)=K\cap u^{\perp}
\right\}.
\end{equation}
The hypothesis makes every $\mathcal P_u$ nonempty.  Right composition by $G$ is free and transitive on each fiber.

\subsection{Local exact sections}\label{subsec:local}

\begin{proposition}[Lipschitz isometry bundle]\label{prop:lipschitz-bundle}
The disjoint union
\[
\mathcal P=\bigsqcup_{u\in S(V)}\mathcal P_u
\]
admits the structure of a principal $G$-bundle over $S(V)$ with a finite Lipschitz atlas.
\end{proposition}

\begin{proof}
Fix $u_*\in S(V)$ and a complex unitary map $R_*:E\to u_*^{\perp}$.  For $u$ near $u_*$, let
\[
P_uv=v-\langle v,u\rangle u
\]
be Hermitian orthogonal projection onto $u^{\perp}$ and define
\[
D_u=R_*^*P_uR_*:E\to E,
\qquad
J_u=P_uR_*D_u^{-1/2}:E\to u^{\perp}.
\]
Since $D_{u_*}=I_E$, after shrinking the neighborhood $D_u$ is positive definite and $J_u$ depends smoothly on $u$.  Directly,
\[
J_u^*J_u=D_u^{-1/2}R_*^*P_uR_*D_u^{-1/2}=I_E,
\]
so $J_u$ is unitary.

Pull the section back to $E$:
\[
T_u=J_u^{-1}(K\cap u^{\perp})
=\{x\in E:p_K(J_ux)\leq1\}.
\]
The body $T_u$ is balanced: $u^{\perp}$ is a complex subspace, $K\cap u^{\perp}$ is balanced, and $J_u$ is complex linear.  Lemma~\ref{lem:covariance} therefore shows that $C(T_u)$ commutes with $J$, so $C(T_u)^{\pm1/2}$ are complex linear.  This observation is what keeps the normalization below inside the complex-linear category.

Because $K$ contains and is contained in fixed auxiliary Euclidean balls, \eqref{eq:gauge-lip} and smoothness of $J_u$ show that $\rho_{T_u}$ depends locally Lipschitzly on $u$, uniformly on the unit sphere of $E$.  Polar-coordinate formulas then show that $\vol(T_u)$, $C(T_u)$ and each selected moment $M_{j_i}(T_u)$ are locally Lipschitz in $u$.  The spectra of $C(T_u)$ stay in a compact subset of the positive cone on a sufficiently small chart, so $C(T_u)^{\pm1/2}$ also vary locally Lipschitzly; compare \cite[Chapter~1]{Bhatia2007}.

Normalize
\[
\widetilde T_u=C(T_u)^{-1/2}T_u.
\]
Since $T_u$ is complex-linearly equivalent to $S$ and both $\widetilde T_u$ and $S$ have covariance $I_E$, there exists $Q_u\in U(E)$ with
\[
\widetilde T_u=Q_uS.
\]
Indeed, if $T_u=LS$, then $C(T_u)=LL^*$ and $C(T_u)^{-1/2}L$ is unitary.  Lemmas~\ref{lem:moment-detect} and \ref{lem:orbit} now recover the coset $Q_uG$ locally Lipschitzly from $\mathbf M(\widetilde T_u)$.  A smooth local section of the quotient map $U(E)\to U(E)/G$ gives, after shrinking the chart, a locally Lipschitz choice of $Q_u$ itself.

Define
\begin{equation}\label{eq:local-exact-map}
A_u=J_u C(T_u)^{1/2}Q_u.
\end{equation}
Then $A_u(E)=u^{\perp}$ and
\[
A_u(S)=J_u C(T_u)^{1/2}\widetilde T_u=J_uT_u=K\cap u^{\perp},
\]
so $A_u\in\mathcal P_u$.  Thus \eqref{eq:local-exact-map} is a local Lipschitz section.

If $A_i,A_j$ are two such sections, then on an overlap
\[
g_{ij}(u)=\bigl(A_i(u)^*A_i(u)\bigr)^{-1}A_i(u)^*A_j(u)
\]
is exactly $A_i(u)^{-1}A_j(u)$ as a map $E\to E$, hence lies in $G$.  On relatively compact charts the least singular value of $A_i$ is bounded away from zero, so $g_{ij}$ is Lipschitz.  Compactness of $S(V)$ yields a finite cover of such charts, which defines the claimed Lipschitz principal bundle.
\end{proof}

\subsection{Reduction to the identity component}\label{subsec:connected}

Let $G^{\circ}$ be the identity component of $G$.

\begin{lemma}[connected reduction]\label{lem:connected-reduction}
The bundle $\mathcal P$ contains a principal $G^{\circ}$-subbundle $\mathcal P^{\circ}\to S^{2n+1}$ with a finite Lipschitz atlas.
\end{lemma}

\begin{proof}
The quotient $G/G^{\circ}$ is finite, and therefore
\[
\mathcal P/G^{\circ}\longrightarrow S^{2n+1}
\]
is a finite covering.  Since $S^{2n+1}$ is simply connected for $n\geq1$, each connected component of this covering maps homeomorphically to the base.  Choose one component $\Sigma$ and take its inverse image in $\mathcal P$; this is the desired $G^{\circ}$-subbundle.

To retain the Lipschitz atlas, choose the chart domains in Proposition~\ref{prop:lipschitz-bundle} connected and write the original local sections as $s_i$, with
\[
s_j=s_i g_{ij}.
\]
On $U_i$, the chosen component $\Sigma$ corresponds to a fixed coset $g_iG^{\circ}\in G/G^{\circ}$.  Thus
\[
s_i^{\circ}=s_i g_i
\]
takes values in $\mathcal P^{\circ}$.  On an overlap,
\[
s_j^{\circ}=s_i^{\circ}g_{ij}^{\circ},
\qquad
g_{ij}^{\circ}=g_i^{-1}g_{ij}g_j.
\]
Because both $s_i^{\circ}$ and $s_j^{\circ}$ lie in the same $G^{\circ}$-subbundle, $g_{ij}^{\circ}$ takes values in $G^{\circ}$.  Constant left and right multiplication preserve Lipschitz regularity, so the maps $g_{ij}^{\circ}$ form a finite Lipschitz atlas for $\mathcal P^{\circ}$.
\end{proof}

\begin{proposition}[positive-degree trivialization]\label{prop:degree-pullback}
There exist an integer $d\geq1$ and a smooth map
\[
\phi:S^{2n+1}\longrightarrow S^{2n+1},
\qquad \deg\phi=d,
\]
such that $\phi^*\mathcal P^{\circ}$ is topologically trivial.
\end{proposition}

\begin{proof}
By \eqref{eq:bundle-classification}, the bundle determines a class
\[
\alpha\in\pi_{2n+1}(BG^{\circ})\cong\pi_{2n}(G^{\circ}).
\]
The index $2n$ is positive and even, so Lemma~\ref{lem:finite-homotopy} implies that this group is finite.  Choose $d\geq1$ with $d\alpha=0$.  Represent the element $d\in\pi_{2n+1}(S^{2n+1})\cong\Z$ by a continuous self-map $\phi_0$ of degree $d$.  A sufficiently close smooth approximation $g:S^{2n+1}\to V_{\R}$ is nowhere zero; after radial normalization, $\phi=g/|g|$ is smooth and homotopic to $\phi_0$, hence still has degree $d$.  Pullback by $\phi$ multiplies the class by $d$, so
\[
[\phi^*\mathcal P^{\circ}]=d\alpha=0.
\]
The pullback bundle is therefore trivial.
\end{proof}

This is the point at which the complex field removes the parity obstruction in the real argument: the relevant homotopy group is $\pi_{2n}(G^{\circ})$, even for every $n$.

\subsection{Lipschitz regularization without losing exactness}\label{subsec:regularization}

The bundle in Proposition~\ref{prop:degree-pullback} has a continuous global section, but the degree calculation in Section~\ref{sec:degree} needs a Lipschitz one.  A generic approximation in the ambient space of linear maps would destroy exactness.  We therefore approximate in an associated vector bundle and retract equivariantly back to the orbit, following the regularization principle of \cite[Lemmas~3.7--3.9]{LuYang2026}.

\begin{lemma}[Lipschitz approximation in vector bundles]\label{lem:vector-approx}
Let $\mathcal E\to Z$ be a finite-dimensional real vector bundle over a compact smooth manifold.  Suppose $\mathcal E$ has a finite atlas whose transition maps are Lipschitz and fiberwise orthogonal.  Every continuous section of $\mathcal E$ can be approximated uniformly by Lipschitz sections.
\end{lemma}

\begin{proof}
Let $\sigma_0$ be continuous and choose a finite trivializing cover $\{U_i\}$.  Take a smooth partition of unity $\{\psi_i\}$ with $\supp\psi_i\Subset U_i$.  In the $i$th chart, write $\sigma_0$ as a continuous map $f_i:U_i\to W$ into the Euclidean fiber.  By Whitney approximation for functions \cite[Theorem~6.21]{Lee2013}, choose smooth $g_i$ uniformly close to $f_i$ on $\supp\psi_i$.  The local section represented by $\psi_i g_i$ extends by zero to a global Lipschitz section because its support stays away from $\partial U_i$.  In another chart it is multiplied by a Lipschitz orthogonal transition matrix, so it remains Lipschitz.  Summing these global sections gives a Lipschitz section $\sigma$ and
\[
\|\sigma(z)-\sigma_0(z)\|
\leq\sum_i\psi_i(z)|g_i(z)-f_i(z)|,
\]
which can be made uniformly arbitrarily small.
\end{proof}

\begin{lemma}[equivariant Lipschitz tubular retraction]\label{lem:equiv-retract}
Let $L$ be a compact subgroup of the orthogonal group of a finite-dimensional Euclidean space.  Suppose $L$ is embedded in a finite-dimensional Euclidean matrix space $W$ on which left multiplication by $L$ is orthogonal.  Then there is an $L$-invariant open neighborhood $\mathcal U$ of $L$ and an $L$-equivariant Lipschitz retraction
\[
q:\mathcal U\to L,
\qquad
q(\ell T)=\ell q(T).
\]
The map $q$ may moreover be chosen smooth on $\mathcal U$.
\end{lemma}

\begin{proof}
By the tubular neighborhood theorem \cite[Chapter~10]{Lee2013}, the compact embedded submanifold $L\subset W$ has a smooth normal projection $q_0:\mathcal U_0\to L$ on a tubular neighborhood $\mathcal U_0$.  Since left multiplication by $L$ is orthogonal and preserves $L$, it preserves the normal bundle.  After shrinking the tube uniformly around the compact set $L$, we may therefore assume that $\mathcal U_0$ is $L$-invariant and that
\[
q_0(\ell T)=\ell q_0(T).
\]

Choose a smaller $L$-invariant tubular neighborhood $\mathcal U$ whose closure is compactly contained in $\mathcal U_0$.  Smoothness gives a uniform bound for $Dq_0$ on a neighborhood of $\overline{\mathcal U}$.  Cover $\overline{\mathcal U}$ by finitely many Euclidean balls on which the mean-value estimate gives one common local Lipschitz constant, and let $\delta>0$ be a Lebesgue number for this cover.  If $|T-T'|<\delta$, the two points lie in one such ball and the local estimate applies.  If $|T-T'|\geq\delta$, compactness of $L$ gives
\[
|q_0(T)-q_0(T')|
\leq \operatorname{diam}(L)
\leq \frac{\operatorname{diam}(L)}{\delta}|T-T'|.
\]
Thus $q=q_0|_{\mathcal U}$ is globally Lipschitz on $\mathcal U$, while retaining smoothness, equivariance, and the retraction property.
\end{proof}

\begin{proposition}[exact Lipschitz section]\label{prop:exact-lipschitz}
Let $P\to Z$ be a topologically trivial principal bundle over a compact smooth manifold, with compact matrix structure group $L$ and a finite Lipschitz atlas.  Assume that $L$ is embedded in a Euclidean matrix space $W$ so that left multiplication by $L$ is orthogonal.  Then $P$ has a Lipschitz global section.
\end{proposition}

\begin{proof}
Let $W$ be a Euclidean matrix space containing $L$ and form the associated vector bundle
\[
\mathcal E=P\times_L W,
\qquad
(p\ell,T)\sim(p,\ell T).
\]
The map
\[
\iota:P\longrightarrow\mathcal E,
\qquad
p\longmapsto[p,I],
\]
identifies $P$ with the orbit subbundle $P\times_LL\subset\mathcal E$.  Because its transition functions act on $W$ by left multiplication and that representation of $L$ is orthogonal, the associated vector bundle satisfies the hypotheses of Lemma~\ref{lem:vector-approx}.

Topological triviality gives a continuous section $s_0$ of $P$.  Approximate $\iota\circ s_0$ uniformly by a Lipschitz section $\sigma$ of $\mathcal E$, close enough that every value lies in the fiberwise copy of the invariant neighborhood from Lemma~\ref{lem:equiv-retract}.  Equivariance makes
\[
Q([p,T])=[p,q(T)]
\]
well defined: from $(p\ell,T)\sim(p,\ell T)$ one gets
\[
[p\ell,q(T)]=[p,\ell q(T)]=[p,q(\ell T)].
\]
Because $q$ is Lipschitz and the transition functions act orthogonally, $Q$ is a Lipschitz bundle map in every chart.  Its image is the orbit subbundle $\iota(P)$, so $\iota^{-1}Q\sigma$ is a Lipschitz global section of $P$.
\end{proof}

We now apply this to the pullback of $\mathcal P^{\circ}$.

\begin{theorem}[global exact family]\label{thm:global-family}
There exist an integer $d\geq1$, a smooth map $\phi:S^{2n+1}\to S^{2n+1}$ of degree $d$, and a Lipschitz map
\[
z\longmapsto A_z\in\Hom_{\C}(E,V)
\]
such that
\begin{equation}\label{eq:global-exact}
A_z(E)=\phi(z)^{\perp},
\qquad
A_z(S)=K\cap\phi(z)^{\perp}
\end{equation}
for every $z\in S^{2n+1}$.
\end{theorem}

\begin{proof}
Choose $d$ and $\phi$ as in Proposition~\ref{prop:degree-pullback}.  We apply Proposition~\ref{prop:exact-lipschitz} with $L=G^{\circ}$ and $W=\operatorname{End}_{\R}(E_{\R})$ equipped with the Hilbert--Schmidt inner product.  Since $G^{\circ}\subset U(E)$, left multiplication by $G^{\circ}$ is orthogonal on $W$.  Let $s_i^{\circ}:U_i\to\mathcal P^{\circ}$ be the finite Lipschitz atlas from Lemma~\ref{lem:connected-reduction}.  On the pullback chart $\phi^{-1}(U_i)$ the natural local section is
\[
\widetilde s_i(z)=(z,s_i^{\circ}(\phi(z))).
\]
Since $\phi$ is smooth on the compact sphere, it is Lipschitz; hence the pulled-back transition functions and the ambient maps $s_i^{\circ}(\phi(z))\in\Hom_{\C}(E,V)$ are Lipschitz.

Proposition~\ref{prop:exact-lipschitz} gives a Lipschitz global section $s$ of $\phi^*\mathcal P^{\circ}$.  By definition of a Lipschitz section, on $\phi^{-1}(U_i)$ there is a Lipschitz map $h_i:\phi^{-1}(U_i)\to G^{\circ}$ such that
\[
s(z)=\widetilde s_i(z)h_i(z).
\]
Forget the pullback coordinate and write the corresponding exact map as
\[
A_z=s_i^{\circ}(\phi(z))\,h_i(z)\in\Hom_{\C}(E,V).
\]
Both factors are bounded and Lipschitz on each chart, so $z\mapsto A_z$ is locally Lipschitz in the ambient operator norm.  A finite cover of the compact sphere has a Lebesgue number; combining the local bounds for nearby points with the boundedness of $A_z$ for points farther apart gives one global Lipschitz constant.  Thus $z\mapsto A_z$ is an ambient Lipschitz map.  Since $s(z)$ lies in the exact fiber over $\phi(z)$, the identities \eqref{eq:global-exact} hold, and complex linearity is automatic.
\end{proof}

The theorem is the point at which the finite homotopy obstruction becomes useful rather than merely removable.  The family $A_z$ is constructed only after reparametrizing the normal sphere by $\phi$, but the nonzero degree of that reparametrization is retained.  Section~\ref{sec:degree} shows that this integer is exactly the multiplicity with which the image fills the dilated body in the signed degree identity.

\begin{remark}\label{rem:no-projective}
No equivariance of $\phi$ under the circle action on $S^{2n+1}$ is required.  Although $u$ and $e^{it}u$ determine the same hyperplane, the proof works on the unit-normal sphere itself.  The fiber over $z$ in the pullback is simply the exact-isometry fiber belonging to $\phi(z)^{\perp}$.
\end{remark}

\section{From degree to polynomiality}\label{sec:degree}

Let $p_S$ denote the Minkowski functional of the normalized model $S$.  Exactness gives, for every $z$ and $y\in E$,
\begin{equation}\label{eq:gauge-exact}
p_K(A_zy)=p_S(y).
\end{equation}
Put
\[
m=\dim_{\R}V_{\R}=2n+2.
\]
Thus the parameter sphere in Theorem~\ref{thm:global-family} is $S^{m-1}$.

\subsection{The boundary map and its orientation}\label{subsec:boundary}

Fix $0\neq y\in E$ and write
\[
r=p_S(y)>0.
\]
Define
\[
f_y:S^{m-1}\to r\partial K,
\qquad
f_y(z)=A_zy,
\]
and
\[
w_y:S^{m-1}\to S^{m-1},
\qquad
w_y(z)=\frac{A_zy}{|A_zy|}.
\]
The first map lands in $r\partial K$ by \eqref{eq:gauge-exact}.

\begin{lemma}[boundary degree]\label{lem:boundary-degree}
Give $r\partial K$ the boundary orientation induced from the oriented real vector space $V_{\R}$.  Then
\[
\deg w_y=d,
\qquad
\deg f_y=d.
\]
\end{lemma}

\begin{proof}
By \eqref{eq:global-exact}, $A_zy\in\phi(z)^{\perp}$.  Hence $w_y(z)$ and $\phi(z)$ are Hermitian orthogonal unit vectors, and therefore orthogonal in $V_{\R}$.  The formula
\[
H(z,t)=
\cos\!\left(\frac{\pi t}{2}\right)\phi(z)
+
\sin\!\left(\frac{\pi t}{2}\right)w_y(z)
\]
defines a homotopy in $S^{m-1}$ from $\phi$ to $w_y$.  Thus $\deg w_y=\deg\phi=d$.

Consider the radial homeomorphism
\[
R_r:S^{m-1}\to r\partial K,
\qquad
R_r(\theta)=r\rho_K(\theta)\theta,
\]
and its extension
\[
\widehat R_r(0)=0,
\qquad
\widehat R_r(t\theta)=tr\rho_K(\theta)\theta.
\]
To determine its orientation, put $a(\theta)=r\rho_K(\theta)>0$ and define
\[
\Psi_s(t\theta)=t\bigl((1-s)+sa(\theta)\bigr)\theta,
\qquad 0\leq s\leq1.
\]
For each $s$, the radial factor is strictly positive, so $\Psi_s$ is a homeomorphism.  This is an isotopy from the identity to $\widehat R_r$.  Therefore $\widehat R_r$ is orientation preserving, and its boundary restriction $R_r$ has degree $+1$ with the boundary orientation on $r\partial K$.

Finally, \eqref{eq:gauge-exact} gives
\[
f_y=R_r\circ w_y.
\]
Multiplicativity of degree yields $\deg f_y=d$.
\end{proof}

The isotopy in this proof is important: it fixes the sign needed later in the signed degree formula, rather than defining an orientation after the fact.

\subsection{The Lipschitz cone}\label{subsec:cone}

Extend $z\mapsto A_z$ homogeneously to
\[
\mathcal A:B^m\to\Hom_{\C}(E,V),
\qquad
\mathcal A(0)=0,
\quad
\mathcal A(tz)=tA_z.
\]

\begin{lemma}[Lipschitz homogeneous extension]\label{lem:cone-lipschitz}
The map $\mathcal A$ is Lipschitz on $B^m$.
\end{lemma}

\begin{proof}
Let
\[
M_0=\sup_{z\in S^{m-1}}\|A_z\|,
\qquad
M_1=\Lip(z\mapsto A_z).
\]
Write $x=rz$ and $x'=sw$ with $0\leq r\leq s\leq1$.  Then
\[
\|rA_z-sA_w\|
\leq rM_1|z-w|+M_0|r-s|.
\]
Moreover $|r-s|\leq|x-x'|$, and
\[
r|z-w|
\leq |rz-rw|+|rw-sw|
\leq |x-x'|+|r-s|
\leq2|x-x'|.
\]
Hence
\[
\|\mathcal A(x)-\mathcal A(x')\|
\leq(2M_1+M_0)|x-x'|.
\]
The same estimate covers the case where one point is the origin.
\end{proof}

For fixed $y\neq0$ define
\[
F_y:B^m\to V_{\R},
\qquad
F_y(x)=\mathcal A(x)y.
\]
This map is Lipschitz and satisfies
\[
F_y(B^m)\subset rK,
\qquad
F_y(S^{m-1})\subset r\partial K.
\]

\begin{lemma}[degree of the cone]\label{lem:cone-degree}
For $v\notin r\partial K$,
\[
\deg(F_y,\Int B^m,v)=
\begin{cases}
d,&v\in\Int(rK),\\
0,&v\notin rK.
\end{cases}
\]
\end{lemma}

\begin{proof}
Let $C_{w_y}$ be the ordinary cone on $w_y$.  From the definition of $w_y$ and \eqref{eq:gauge-exact},
\[
|A_zy|=r\rho_K(w_y(z)).
\]
Therefore
\begin{equation}\label{eq:cone-factorization}
F_y=\widehat R_r\circ C_{w_y}.
\end{equation}
The homeomorphism $\widehat R_r$ is orientation preserving by Lemma~\ref{lem:boundary-degree}.  If $v\in\Int(rK)$, put $v_0=\widehat R_r^{-1}(v)\in\Int B^m$.  The composition rule for Brouwer degree, \eqref{eq:cone-degree-standard}, and $\deg w_y=d$ give
\[
\deg(F_y,\Int B^m,v)
=
\deg(C_{w_y},\Int B^m,v_0)
=d.
\]
Outside $rK$ the degree is zero because $F_y(B^m)\subset rK$.
\end{proof}

\subsection{The signed degree formula}\label{subsec:polynomiality}

\begin{proposition}[polynomial powers of the model norm]\label{prop:polynomial-powers}
For every integer $j\geq0$, the function
\[
y\longmapsto p_S(y)^{m+2j}
\]
is a homogeneous real polynomial of degree $m+2j$ on $E_{\R}$.
\end{proposition}

\begin{proof}
Fix $y\neq0$ and set $r=p_S(y)$.  Let $\Omega=\Int B^m$.  By Lemma~\ref{lem:cone-lipschitz}, $F_y$ is Lipschitz on $\overline\Omega$.  Extend it coordinatewise to a Lipschitz map on $V_{\R}$ using McShane's theorem \cite{McShane1934}.  We keep the notation $F_y$ for the extension.  Since
\[
F_y(\partial\Omega)\subset r\partial K
\]
and $r\partial K$ has $m$-dimensional Lebesgue measure zero, the signed degree formula \eqref{eq:signed-degree} applies.  Indeed, \eqref{eq:gauge-lip} makes the radial parametrization $S^{m-1}\to r\partial K$ Lipschitz.  Covering $S^{m-1}$ by $O(\varepsilon^{-(m-1)})$ balls of radius $\varepsilon$ shows that its Lipschitz image can be covered by the same order of balls of radius $O(\varepsilon)$; the total $m$-dimensional volume is therefore $O(\varepsilon)$ and tends to zero.

Choose a compactly supported continuous cutoff $\chi$ equal to one on a neighborhood of $rK$ and set
\[
g(v)=\chi(v)|v|^{2j}.
\]
Lemma~\ref{lem:cone-degree} gives
\begin{align}
\int_{B^m}|F_y(x)|^{2j}\det DF_y(x)\,dx
&=d\int_{rK}|v|^{2j}\,dv \notag\\
&=d\,r^{m+2j}\int_K|v|^{2j}\,dv.\label{eq:degree-integral}
\end{align}
The constant on the right is positive apart from the factor $r^{m+2j}$.

It remains to identify the left side as a polynomial in $y$.  Regard
\[
\mathcal A:B^m\to\Hom_{\R}(E_{\R},V_{\R})
\]
as a Lipschitz map into a finite-dimensional Euclidean space.  There is a null set, independent of $y$, outside which $\mathcal A$ is differentiable.  At such a point $x$,
\[
DF_y(x)[h]=(D\mathcal A(x)[h])y.
\]
After choosing real bases, each column of the $m\times m$ matrix $DF_y(x)$ is real-linear in the $2n$ real coordinates of $y$.  Hence
\[
y\longmapsto\det DF_y(x)
\]
is homogeneous polynomial of degree $m$.  Also
\[
y\longmapsto |F_y(x)|^{2j}
=\langle\mathcal A(x)y,\mathcal A(x)y\rangle_{\R}^j
\]
is homogeneous polynomial of degree $2j$.  Their product is therefore homogeneous of degree $m+2j$.

The map $\mathcal A$ is bounded, while $D\mathcal A$ is essentially bounded by its Lipschitz constant.  Thus every coefficient of this polynomial is dominated by an integrable constant on $B^m$, and coefficientwise integration is legitimate.  The left-hand side of \eqref{eq:degree-integral} is consequently a homogeneous polynomial $P_j(y)$ of degree $m+2j$.  Equation \eqref{eq:degree-integral} reads
\[
P_j(y)
=
\left(d\int_K|v|^{2j}\,dv\right)p_S(y)^{m+2j}
\]
for $y\neq0$.  Both sides vanish at the origin, so the identity holds everywhere.  Division by the positive constant proves the proposition.
\end{proof}

\section{Quadratic and Hermitian rigidity}\label{sec:hermitian}

Put $a=n+1$, so $m=2a$.  Proposition~\ref{prop:polynomial-powers} with $j=0,1$ gives nonzero homogeneous real polynomials
\[
P(y)=p_S(y)^{2a},
\qquad
Q(y)=p_S(y)^{2(a+1)}.
\]
Pointwise, and therefore as a polynomial identity,
\begin{equation}\label{eq:ufd-identity}
Q^a=P^{a+1}.
\end{equation}

\begin{lemma}[quadratic rigidity]\label{lem:quadratic-rigidity}
There is a positive-definite real quadratic form $q$ on $E_{\R}$ such that
\[
q(y)=p_S(y)^2
\]
for all $y\in E$.
\end{lemma}

\begin{proof}
The real polynomial ring on $E_{\R}$ is a unique factorization domain; see, for example, \cite[Section~9.3, Theorem~7]{DummitFoote2004}.  For each irreducible factor $\pi$, write $\nu_{\pi}(P)$ and $\nu_{\pi}(Q)$ for its multiplicities.  Equation \eqref{eq:ufd-identity} gives
\[
a\nu_{\pi}(Q)=(a+1)\nu_{\pi}(P).
\]
Since $a$ and $a+1$ are coprime, there is an integer $\ell_{\pi}\geq0$ with
\[
\nu_{\pi}(P)=a\ell_{\pi},
\qquad
\nu_{\pi}(Q)=(a+1)\ell_{\pi}.
\]
Thus every irreducible factor of $P$ occurs in $Q$ with at least the same multiplicity, and $P$ divides $Q$.  The quotient
\[
q=Q/P
\]
is a homogeneous polynomial of degree two.  For $y\neq0$,
\[
q(y)=\frac{p_S(y)^{2(a+1)}}{p_S(y)^{2a}}=p_S(y)^2,
\]
and homogeneity gives the equality at $0$.  Because $p_S$ is a norm, $q(y)>0$ for $y\neq0$.
\end{proof}

A real quadratic form is not yet the desired conclusion: we must recover the complex structure.

\begin{lemma}[Hermitian return]\label{lem:hermitian-return}
The quadratic form $q$ in Lemma~\ref{lem:quadratic-rigidity} is the diagonal of a positive-definite Hermitian form on $E$.  Consequently $S$ is a complex ellipsoid.
\end{lemma}

\begin{proof}
Let $B$ be the symmetric real bilinear form obtained by polarizing $q$:
\[
B(x,y)=\frac{q(x+y)-q(x)-q(y)}{2}.
\]
Since $S$ is the unit ball of a complex norm,
\[
p_S(iy)=p_S(y),
\]
and therefore $q(Jy)=q(y)$.  Polarization gives
\begin{equation}\label{eq:J-bilinear}
B(Jx,Jy)=B(x,y),
\qquad
B(Jx,y)=-B(x,Jy).
\end{equation}
Define
\[
h(x,y)=B(x,y)-iB(Jx,y).
\]
The identities \eqref{eq:J-bilinear} show that $h$ is complex linear in its first variable and conjugate linear in its second.  Symmetry of $B$ gives $h(y,x)=\overline{h(x,y)}$.  Moreover $B(Jy,y)=0$, so
\[
h(y,y)=B(y,y)=q(y)=p_S(y)^2>0
\]
for $y\neq0$.  Thus $h$ is positive-definite Hermitian and
\[
S=\{y:h(y,y)\leq1\}.
\]
\end{proof}

We can now finish the finite-dimensional theorem.

\begin{proof}[Proof of Theorem~\ref{thm:hyperplane}]
By Lemma~\ref{lem:hermitian-return}, the normalized model section $S$ is a complex ellipsoid.  Since every hyperplane section $K\cap H$ is complex-linearly equivalent to $S$, each restricted norm on a hyperplane is induced by a Hermitian inner product.  Hence, for $x,y$ in a common complex hyperplane,
\begin{equation}\label{eq:parallelogram}
p_K(x+y)^2+p_K(x-y)^2
=2p_K(x)^2+2p_K(y)^2.
\end{equation}

Now take arbitrary $x,y\in V$.  Their complex span has dimension at most two.  Since $n\geq2$ and $\dim_{\C}V=n+1$, this span is contained in an $n$-dimensional complex hyperplane.  Thus \eqref{eq:parallelogram} holds for every pair $x,y\in V$.  By the Jordan--von Neumann characterization \cite{JordanVonNeumann1935}, the norm $p_K$ comes from an inner product; in the complex case the standard complex polarization identity yields a Hermitian inner product.  Therefore $K$ is a Hermitian ellipsoid.
\end{proof}

\section{The ambient Banach space}\label{sec:ambient}

\begin{proof}[Proof of Corollary~\ref{cor:main}]
Let $x,y\in X$.  Since $2\leq n<\dim_{\C}X$, the complex span of $x$ and $y$ is contained in some $(n+1)$-dimensional complex subspace $Y\subset X$.  Every complex hyperplane of $Y$ is an $n$-dimensional subspace of $X$, so the hypothesis says that these hyperplanes, with their restricted norms, are mutually complex-linearly isometric.  Theorem~\ref{thm:hyperplane} implies that $Y$ is a complex Hilbert space.  Hence the parallelogram identity holds for the original pair $x,y$.

Since $x$ and $y$ were arbitrary, the norm of $X$ satisfies the parallelogram identity everywhere.  The Jordan--von Neumann polarization formula \cite{JordanVonNeumann1935} therefore defines a Hermitian inner product inducing the norm.  Thus $X$ is a complex Hilbert space.
\end{proof}

\begin{remark}[comparison with the real argument]\label{rem:final}
Lu and Yang's real proof uses the finiteness of $\pi_{n-1}(G^{\circ})$ when the section dimension $n$ is odd and obtains the polynomial powers $p_S^{n+1}$ and $p_S^{n+3}$ \cite[Sections~3--4]{LuYang2026}.  In the complex case the normal sphere has real dimension $2n+1$, so the obstruction group is $\pi_{2n}(G^{\circ})$ and the first two powers are $p_S^{2n+2}$ and $p_S^{2n+4}$.  Both parity requirements are therefore automatic.  The additional final step is the passage from the resulting real quadratic form to a Hermitian form using the $S^1$-invariance of the complex norm.
\end{remark}

\section*{Statement and declaration}
The work is assisted by TARS agent system via exploratory reasoning. X. Dai supplemented critical argument details, refined the manuscript logic, and completed the writing.

\end{document}